\documentclass[12pt]{amsart}
\usepackage{amsmath,amssymb,bm}
\usepackage{latexsym}
\usepackage[dvipdfmx]{graphicx}
\usepackage{graphicx}
\usepackage{float}

\newtheorem{Thm}{Theorem}[section]
\newtheorem{Lem}[Thm]{Lemma}
\newtheorem{Prop}[Thm]{Proposition}

\newtheorem*{xrem}{Remark}
\newtheorem*{xque}{Question}
\newtheorem*{xexa}{Example}

\makeatletter
\@namedef{subjclassname@2020}{%
  \textup{2020} Mathematics Subject Classification}
\makeatother

\newcommand{\R}{\mathbb{R}}
\newcommand{\CC}{\mathbb{C}}

\newcommand{\B}{\mathbb{B}}
\newcommand{\Z}{\mathbb{Z}}
\newcommand{\N}{\mathbb{N}}

\newcommand{\F}{\mathcal{F}}
\newcommand{\G}{\mathcal{G}}
\newcommand{\HH}{\mathcal{H}}
\newcommand{\C}{\mathcal{C}}
\newcommand{\E}{\mathcal{E}}

\bmdefine{\boldr}{r}
\bmdefine{\boldv}{v}

\begin{document}

\baselineskip=17pt

\title
{Foliations on the open $3$-ball by complete surfaces}

\author{Takashi Inaba}
\address{
Department of Mathematics and Informatics\\
Graduate School of Science\\
Chiba University\\
Chiba 263-8522\\
Japan
}
\email{inaba@math.s.chiba-u.ac.jp}

\author{Kazuo Masuda}
\email{math21@maple.ocn.ne.jp}

\date{}

\begin{abstract}
When is a manifold a leaf of a complete closed foliation on the open unit ball? 
We give some answers to this question.
\end{abstract}



\maketitle

\section{Introduction and statement of results}
This paper is concerned with the topology of leaves of foliations.
The concept of a foliation appeared in 1940's
as a geometric approach 
to solutions of differential equations,
and it is now widespread among various areas such as
complex analysis, exterior differential systems and
contact topology ({\it e.g.} \cite{BI,I}).
Recall that
a codimension $q$ $C^r$ {\it foliation} $\F$ 
on an $n$-dimensional smooth manifold $M$ is a decomposition 
$\{L_\lambda\}_{\lambda\in\Lambda}$ of $M$
into a disjoint union of
injectively immersed connected $(n-q)$-dimensional submanifolds $L_\lambda$
satisfying the following local triviality:
each point of $M$ has a neighborhood $U$ such that
$\F$ restricted to $U$ is $C^r$ diffeomorphic to the family 
$\{\R^{n-q}\times\{y\}\}_{y\in\R^q}$ of parallel 
$(n-q)$-dimensional planes in $\R^n$.
$L_\lambda$ is called a {\it leaf} of $\F$.
Note that, by collecting all the vectors tangent to leaves,
the foliation can alternatively be defined as an 
integrable subbundle
of $TM$.

In 1975, Sondow \cite{S} posed a basic question: when is a manifold a leaf?
This question is natural
(because it is regarded as a generalization of the classical embedding problem
in differential topology) and
important (because it may be related to the study of 
the topology of integral manifolds of differential equations).
Thus, since then the question has been investigated
extensively in various possible settings
(see {\it e.g.} \cite{CC, G, HB, HP, INTT, MS}).

The purpose of this paper is to consider the question 
in an interesting new setting. 
Let $\F$ be a foliation on a Riemannian manifold $(M, g)$.
A leaf $L$ of $\F$ is called {\it closed} if $L$ is a closed subset of $M$
(this is equivalent to say that $L$ is {\it properly embedded}\ ),
and {\it complete}  if $L$ is complete with respect to the induced Riemannian
metric $g|L$.
A foliation $\F$ is said to be {\it closed} (resp.  {\it complete})
if all leaves of  $\F$ are closed (resp.  complete).
Now, our setting is as follows: 
We fix, as the manifold which supports foliations,  
the open unit ball $\B^n$ of the Euclidean space $\R^n$
with the induced Euclidean metric --- the simplest incomplete open manifold.
And, foliations we try to construct on $\B^n$ should be complete and closed.
Newness of this setting is to
treat {\it complete} closed foliations
on {\it incomplete} open manifolds.
As one can imagine, 
in order to construct such foliations,
one must \lq\lq turbulize"
all the leaves
along the (ideal) boundary of $\B^n$. 

The motivation of this work comes from recent deep works of Alarc\'{o}n, 
Globevnik and Forstneri\^{c} \cite{A, AF, AG}.
They consider holomorphic foliations on the open ball of $\CC^n$. 
Our work is, in a sense, a real smooth $(C^\infty)$ version of theirs.
Since holomorphic objects are very \lq rigid',
constructions of complete holomorphic foliations are
much harder than those of real ones.
The advantage of our approach is that, by forgetting holomorphic rigidity,
one can concentrate on overcoming purely topological difficulties.
In fact, on some topic we have thus succeeded in constructing
infinitely many examples of foliations that are new in the literature 
(see Theorem 1.2 below).

Now, the first result of this paper is the following, 
whose holomorphic version has been obtained
by Alarc\'{o}n and Globevnik \cite{A, AG}. 
(Note that our result is independent of theirs,
because of the difference of codimension. 
The codimension of our foliation is $1$ 
(the most cramped codimension, see \S 6), 
while the {\it real} codimension of their foliations are at least
$2$.)
\begin{Thm}\label{main}
For any connected open orientable smooth surface $\Sigma$,
there is a codimension $1$ complete closed smooth 
foliation on $\B^3$ with a leaf diffeomorphic to $\Sigma$.
\end{Thm}
\begin{xrem}
{\rm
In \cite{HB}, Hector and Bouma showed the same statement on $\R^3$.
In \cite{HP}, Hector and Peralta-Salas generalized it in higher dimensions.
}
\end{xrem}

\begin{xrem} 
{\rm
The corresponding result to Theorem \ref{main}
for Sondow's original question (i.e. the realization of manifolds as leaves
of foliations on {\it compact} manifolds)
was first obtained by 
Cantwell and Conlon \cite{CC}.
For recent developments in this area, see e.g. \cite{AL, MS}.
}
\end{xrem}

\begin{xrem}
{\rm
A non-orientable surface cannot be a leaf of a foliation of $\B^3$.
In fact, if it can, the foliation must be transversely non-orientable.
The existence of such a foliation contradicts the simply-connectedness of 
$\B^3$.
}
\end{xrem} 

Our next concern is a uni-leaf foliation.
Here, we call a foliation $\F$ {\it uni-leaf}
if all the leaves of $\F$ are mutually diffeomorphic.

\begin{xexa}
{\rm
A complete closed uni-leaf foliation on $\B^2=\{ (x,y)\in \R^2\mid x^2+y^2<1\}$ can easily be constructed as follows.
Begin with the standard foliation $\HH$ on $\B^2$ defined by $dy=0$.
Then, obviously all leaves of $\HH$ are diffeomorphic to the real line
and closed in $\B^2$.
Let $h$ be a diffeomorphism of $\B^2$
defined by 
$\displaystyle h(r,\theta)=\left(r,\theta+\tan\frac{\pi r^2}2\right)$,
where $(r,\theta)$ are the polar coordinates.
Then, $h$ sends any leaf $\ell$ of $\HH$ 
to a complete curve $h(\ell)$ in $\B^2$,
because each end of $h(\ell)$ spirals asymptotically on $\partial\B^2$. 
Hence, $h(\HH)$ is a foliation we have desired.
Note that, since $h$ is real analytic ($C^\omega$), so is $h(\HH)$.
}
\end{xexa}

So, let us consider uni-leaf foliations on $\B^3$.
For a connected open orientable surface $\Sigma$,
let $\E$ be the set of ends of $\Sigma$
with the usual topology
and $\E^*$ the closed subset of $\E$ consisting of nonplanar ends. 
It is known (\cite{R}) that the pair $(\E, \E^*)$ 
and the genus
determine
the homeomorphism type of $\Sigma$.
It is also known that
two smooth surfaces are diffeomorphic if and only if they are homeomorphic.

Now, we will introduce a new concept.
We assume that the genus $g$ of $\Sigma$ is either $0$ or $\infty$.
Let $e$ be a point of $\E$. 
Let $Z$ be 
empty if $g=0$ 
and a countably infinite
subset of $\E-\E^*-\{e\}$ if $g=\infty$.
Suppose further that every point of $Z$ is an isolated point
of $\E$ 
and the derived set of $Z$ in $\E$ is $\E^*$.
In this situation we say that the $4$-tuple
$(\E,\E^*, Z,e)$ satisfies the {\it self-similarity property}
if the following condition holds:
there exist
two copies $(\E^+,\E^{+*}, Z^+, e^+)$, $(\E^-,\E^{-*}, Z^-, e^-)$ 
of $(\E,\E^*,Z,e)$
 and 
a homeomorphism 
$h: \E^+\vee_{e^+=e^-} \E^-\to \E$
such that 
$h(e^\pm)=e$
and
$h(Z^+\sqcup Z^-)=Z$
(hence, $h(\E^{+*}\vee_{e^+=e^-} \E^{-*})=\E^*$), 
where $\vee$ is the wedge sum.

\begin{xrem}
{\rm
This concept is not the same as the usual self-similarity in fractal geometry. 
Ours is a kind of {\it pointed} self-similarity, 
meaning that we fix a basepoint once and for all and then only consider
subspaces containing the basepoint and mappings preserving the basepoint.
}
\end{xrem}

\begin{xexa}
{\rm
$(1)$ Let $C$ be a Cantor set embedded in $S^2$.
Then, all ends of the surface $\Sigma_C=S^2-C$ are planar,
and the endset $\E$ of $\Sigma_C$ is identified with $C$
and hence has the self-similarity property.
In fact, $\E$ can be expressed as $\E=\E^+\vee_{e}\E^-$,
where $\E^+$ and $\E^-$ are subsets of $\E$ both homeomorphic to $C$
such that $\E^+\cap\E^-=\{e\}$ for some $e\in\E$. In this case, $Z$ is empty.

$(2)$
Let $J$ be the orientable open surface with one end and infinite genus
$($so called Jacob's ladder$)$.
We take in $J$ a discrete infinite subset $S$ 
and put $\Sigma_J=J-S$.
Then, the endset $\E$ of $\Sigma_J$ consists of isolated planar ends
$e_n$ $($$n\in\N$$)$ each of which corresponds with a point of $S$
and one nonplanar end $e$ to which $e_n$'s converge.
Thus, $\E=\{e\}\cup\{e_n\}_{n\in\N}$ and $\E^*=\{e\}$.
See Fig.{\rm 1}.
Set $Z=\{e_n\}_{n\in\N}$.
Then, the $4$-tuple $(\E,\E^*,Z,e)$ satisfies the self-similarity property.
Indeed, it suffices to define:
$\E^+=\{e\}\cup\{e_{2n-1}\}_{n\in\N}$, 
$\E^-=\{e\}\cup\{e_{2n}\}_{n\in\N}$,
$\E^{+*}=\E^{-*}=\{e\}$,
$Z^\pm=\E^\pm-\{e\}$,
$e^+=e^-=e$,
and $h$ is the identitiy.
}
\end{xexa}
\vspace{3cm}

\begin{figure}[htbp]
\centering
\includegraphics[height=8cm]{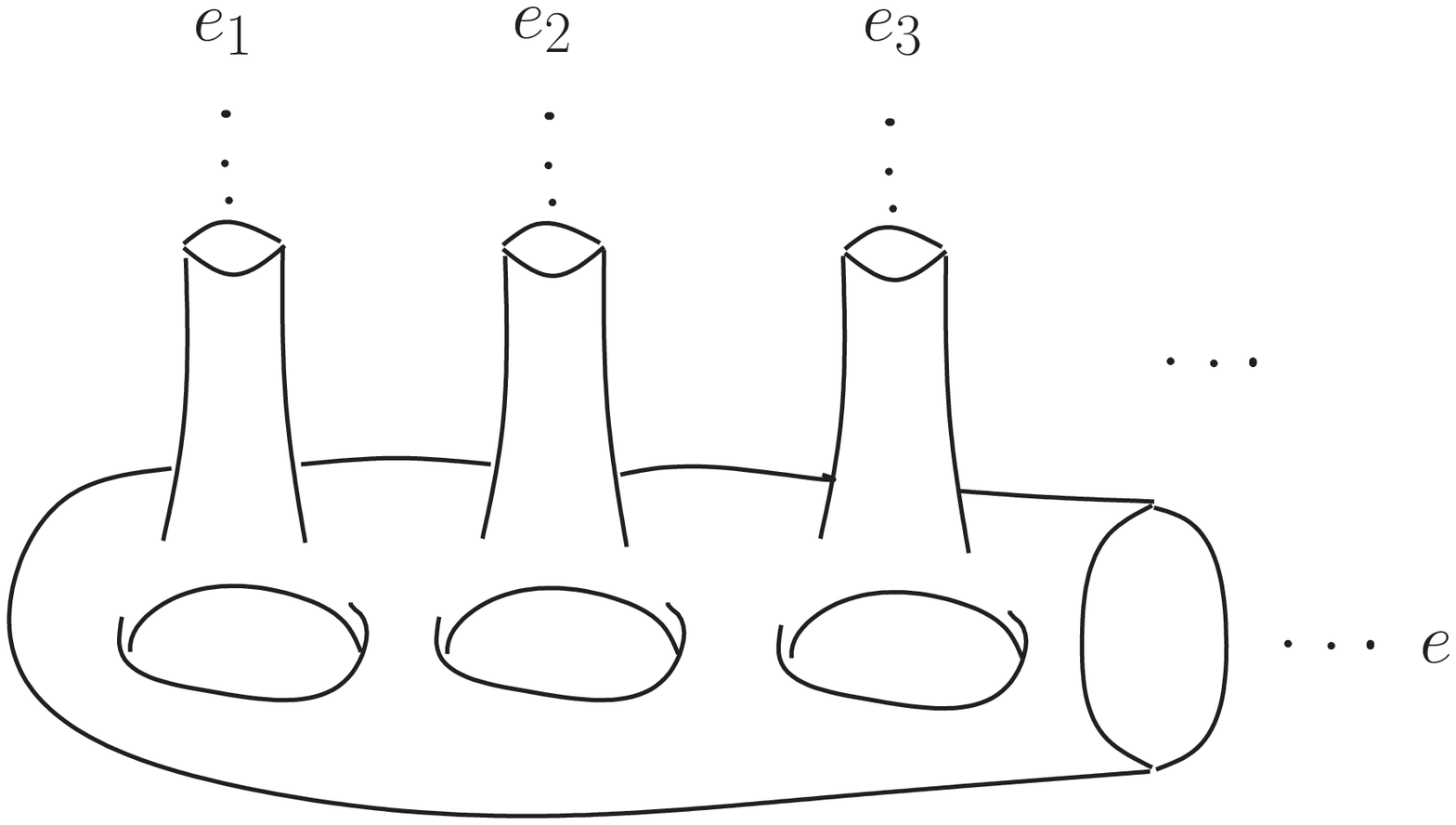}
      \caption{}
  \end{figure}

We remark that there are many other surfaces whose endsets have the self-similarity property.
We will give them at the end of \S 4.

The next theorem shows that infinitely many surfaces can be realized 
as leaves
of uni-leaf foliations on $\B^3$.

\begin{Thm}\label{uni}
Let $\Sigma$ be a connected open orientable smooth surface
with genus $g$ either $0$ or $\infty$,
and let $(\E,\E^*)$ be its endset pair.
Suppose that there exist a point $e$ of $\E$ and 
a subset $Z$ of $\E-\E^*-\{e\}$
such that 
\begin{itemize}
\item[(1)] $Z$ is empty if $g=0$ and countably infinite if $g=\infty$,
\item[(2)] every point of $Z$ is an isolated point
of $\E$, 
\item[(3)] the derived set of $Z$ in $\E$ is $\E^*$, and
\item[(4)] 
$(\E,\E^*, Z,e)$ satisfies the self-similarity property.
\end{itemize}
Then, there exists a codimension $1$ complete closed smooth uni-leaf
foliation of $\B^3$ having $\Sigma$ as a leaf.
\end{Thm}

\begin{xrem}
{\rm
In the holomorphic situation,
the existence of a uni-leaf foliation on the $2$-dimensional
holomorphic ball is known in the case 
where the leaf $\Sigma$ is 
a disk $\{z\in\CC\mid |z|<1\}$ (see \cite{AF}).
It seems that the problem remains open
whether other Rieman surfaces
can be leaves of some holomorphic uni-leaf foliations.
}
\end{xrem} 

The following two arguments are crucial to proving our theorems:

(1) to build a kind of barrier in $\B^n$
(called a {\it labyrinth} in \cite{AF})
in order to force all leaves to become complete.
The existence of a holomorphic labyrinth is a profound result.
On the other hand, we find that a real labyrinth is quite easy to be built (see \S 2).

(2) to show that the self-similarty property of the endset of a surface
is
a sufficient condition for the surface to be a leaf of 
some uni-leaf foliation.
This assertion is proved
by a careful construction of a rather complicated 
submersive function on some domain of $\R^3$ (see \S 4).

We close this section with two more remarks.

\begin{xrem}
{\rm
All foliations in this paper are $C^\infty$.
For foliations in Theorems 1.1 and 1.2 the authors have an idea of raising the diffentiability to $C^\omega$
by using $C^\omega$ approximation of $C^\infty$ Morse functions and diffeomorphisms.
But, at present, they have not written up the proof in full precision yet.
}
\end{xrem} 

\begin{xrem}
{\rm
All foliations in this paper are closed, hence
their holonomy pseudogroups are always trivial.
}
\end{xrem}  

\section{Constructing complete foliations on the ball}
The content of this section is a real smooth version of the argument 
developed in \cite{AF}.

Let $\{r_k\}$ and $\{s_k\}$ be sequences of real numbers
satisfying $0<s_1<r_1<s_2<r_2<\cdots\longrightarrow1$,
and let $B_k$ (resp. $S_k$) be the closed ball (resp. the sphere) in $\R^n$ 
centered at the origin with radius $r_k$ (resp. $s_k$). 
Put $\Gamma_k=S_k-U_\varepsilon(p_k)$,
where $0<\varepsilon\ll s_1$, $p_k=(0,\cdots, 0,(-1)^ks_k)$
and $U_\varepsilon(p_k)$ is the open 
$\varepsilon$-neighborhood of $p_k$ in $\R^n$.
A path $\gamma:[0,\infty)\to M$ in an open manifold $M$ is called {\it divergent} 
if $\gamma(t)$ leaves any compact set of $M$ as $t\to \infty$.
Then, the following is evident.

\begin{Lem}
Every divergent smooth path in $\B^n$ avoiding 
$\bigcup_{k\geqq k_0}\Gamma_k$ $($for some $k_0$$)$ has infinite length.
\end{Lem}

Put $P_k=\R^{n-1}\times[-k,k]$.
Let $\Omega$ be an open set of $\R^n$ diffeomorphic to
$\B^n$ such that its image projected to the last coordinate
of $\R^n$ is unbounded.
Then, we can choose an exhaustive sequence $\{C_k\}_{k\in\N} $ of 
subsets of $\Omega$
satisfying the following properties:
(i) $C_k$ is diffeomorphic to the closed ball, 
(ii) $C_k\subset\operatorname{Int} C_{k+1}$,
(iii) $C_k\subset P_{k+1}$ and
(iv) $C_k-P_k\ne\emptyset$.

\begin{Lem}
There exists a diffeomorphism $\Phi$ from $\Omega$ to $\B^n$
such that for all $k\in\N$
\begin{itemize}
\item[($1_k$)] $\Phi(C_k)=B_k$ and 
\item[($2_k$)] $\Phi(P_k\cap C_k)\cap \Gamma_k=\emptyset$.
\end{itemize}
\end{Lem}

\begin{proof}
We will enlarge the domain of definition inductively.
First, define $\Phi$ on $C_1$ so that it satisfies ($1_1$) and ($2_1$).
This is possible because, by (iv) above, 
$C_1-P_1$ is non-empty and
$\Gamma_1\subset\operatorname{Int} B_1$.
Next, suppose $\Phi$ has already been defined on $C_k$ so as to
satisfy ($1_\ell$) and ($2_\ell$) for $\ell\leqq k$ .
Since $\Gamma_{k+1}\subset B_{k+1}-B_k$
and, by (iv) above, $C_{k+1}-P_{k+1}$ is non-empty,
it is possible to extend the definition of $\Phi$ on $C_{k+1}$
so that $\Phi(C_{k+1}-P_{k+1})\supset\Gamma_{k+1}$
and that $\Phi(C_{k+1}-C_k)=B_{k+1}-B_k$.
Then, we see that the resulting $\Phi$ satisfies ($1_{k+1}$) and ($2_{k+1}$).
Since $\{C_k\}_{k\in\N} $ is an exhausting sequence in $\Omega$, 
this inductive procedure gives a diffeomorphism
from $\Omega$ to $\B^n$, as desired.
\end{proof}

\begin{Lem}
Let $\Phi$ be as in Lemma {\rm 2.2.} 
Then, for all $k\in\N$, 
$\Phi(P_k\cap \Omega)\cap \Gamma_k=\emptyset$.
\end{Lem}

\begin{proof}
Since $\Gamma_k\subset B_k=\Phi(C_k)$,
a point $p$ of $\Omega$ satisfies $\Phi(p)\in\Gamma_k$
only if $p\in C_k$.
Hence, by Lemma 2.2 ($2_k$), 
$\Phi(P_k\cap\Omega)\cap\Gamma_k
=\Phi(P_k\cap C_k)\cap\Gamma_k
=\emptyset
$.
\end{proof}

Let $\G$ be a closed foliation on $\Omega$
(i.e. every leaf of $\G$ is a closed subset of $\Omega$).
Then, the direct image $\F=\Phi(\G)$
is a closed foliation on $\B^n$.
Here, we consider the following property (P) for $\G$:

(P) For any leaf $L$ of $\G$ there exists $k\in \N$ such that 
$L\subset P_k$.

\vspace{2mm}

We recall that a leaf $F$ of $\F$ is complete if and only if
every divergent smooth path in $F$ has infinite length.
The next lemma gives a sufficient condition for the completeness
of the leaves of $\F$.

\begin{Lem}
If $\G$ satisfies the property $({\rm P})$,
then all leaves of $\F$ are complete.
\end{Lem}

\begin{proof}
Suppose $\G$ satisfies (P) and let $F$ be any leaf of $\F$.
Put $L=\Phi^{-1}(F)$.
Since $L$ is a leaf of $\G$,
by (P) there exists $k_L\in \N$ such that 
$L\subset P_{k_L}$. 
Then, noticing Lemma 2.3, 
for any $k\geqq k_L$ we have
$F\cap \Gamma_k=
\Phi(L)\cap \Gamma_k
\subset\Phi(P_{k_L}\cap\Omega)\cap \Gamma_k
\subset\Phi(P_k\cap\Omega)\cap \Gamma_k
=\emptyset$.
Therefore, $F$ does not intersect $\bigcup_{k\geqq k_L}\Gamma_k$,
hence, in particular,
neither does any smooth path on $F$.
This with Lemma 2.1 implies the completeness of the leaves of $\F$.
\end{proof}

We summarize the result obtained in this section as follows.

\begin{Prop} 
Let $\G$ be a closed foliation 
on an open subset $\Omega$ of $\R^n$ diffeomorphic to $\B^n$
such that
\begin{itemize} 
\item[(1)] $\operatorname{pr}_n(\Omega)$ is unbounded, 
where $\operatorname{pr}_n:\R^n\to\R$ is the projection to the $n$-th coordinate, and
\item[(2)] each leaf $L$ of $\G$ verifies the property 
$({\rm P}):$ $\operatorname{pr}_n(L)$ is bounded. 
\end{itemize}
Then, there exists a diffeomorphism $\Phi$ from $\Omega$ to $\B^n$ 
such that $\Phi(\G)$ is a complete closed foliation on $\B^n$.
\end{Prop}

\section{Realizing open surfaces as leaves}
First, we prepare an elementary lemma:

\begin{Prop}\label{push}
Let $W$ be a smooth open manifold and $\C$ be a countable family 
of injective smooth paths $c_k:[0,\infty)\to W$ $(k\in\N)$
such that
\begin{itemize}
\item[(1)] $c_k$ is divergent for each $k$.
\item[(2)] they are pairwise disjoint and the family
$\C$ is locally finite.
\end{itemize}
Then, $W-\bigcup_k c_k([0,\infty))$ is diffeomorphic to $W$.
\end{Prop}

The proof consists in \lq\lq pushing to infinity" each point $c_k(0)$
along the path $c_k$. We give it here for the reader's convenience.

\begin{proof}
 Let $\dim W=n$ and give $W$ an arbitrary Riemannian metric.
 Put $N=D^{n-1}\times[-1,\infty)$ (where $D^{n-1}$ is 
the closed unit disk in $\R^{n-1}$
centered at the origin). 
 Take a nonnegative bounded smooth function
$\lambda:N\to\R$ suc that $\lambda=1$ near $\partial N$
and that for $(p,t)\in N$, 
$\lambda(p,t)=0$ if and only if $p=0$ and $t\in [0,\infty)$.
Define a smooth vector field $V$ on $N$ by 
$V=\lambda\dfrac{\partial}{\partial t}$
and let $\varphi:N\times[0,\infty)\to N$ be the (local) flow
generated by $V$.
Then, the map
$g:N\to N-(\{0\}\times[0,\infty))$
defined as
$g(p,t)=\varphi((p,-1),t+1)$
is a diffeomorphism which is the identity near $\partial N$.
Now,
for each $k$, take a neighborhood $N_k$ of $c_k([0,\infty))$
and a diffeomorphism $u_k:N\to N_k$ so that
(1) $N_k$'s are pairwise disjoint and the family is locally finite,
(2) $u_k(0,t)=c_k(t)$ for $t\in [0,\infty)$, and
(3) the diameter of $u_k(D^{n-1}\times\{t\})$ tends to $0$ as $t\to\infty$.
We then obtain a desired diffeomorphism 
$h:W\to W-\bigcup_k c_k([0,\infty))$ by setting:
$h=u_k\circ g\circ{u_k}^{-1}$ on $N_k$ for any $k$, and is the identity everywhere else.
\end{proof}

Using almost the same argument, we can also show the following 

\begin{Prop} 
Let $M$ be a smooth manifold,
$P$ a subset of $M$,
and $N$ a neighborhood of $P$ in $M$.
Put $W=M\times\R$. 
Let $\ell_p=\{p\} \times(-\infty,a_p]$ $($$p\in P, a_p\in\R$$)$
be a family of vertial lines in $W$.
If the union $\bigcup_{p\in P}\ell_p$ is closed in $W$, 
then $W-\bigcup_{p\in P}\ell_p$ is diffeomorphic to $W$ 
by a fiber $(\{\{ m \} \times \R\}_{m\in M})$-preserving diffeomorphism
which is the identity outside $N\times\R$. 
\end{Prop}
\begin{proof}
It suffices to push $\bigcup_p\ell_p$ to $-\infty$ 
with respect to the $\R$-factor.
To be precise, take a nonnegative bounded smooth function
$\mu:W\to\R$ which vanishes exactly on $\bigcup_p\ell_p$
and is constantly $1$ outside $N\times\R$.
We consider the flow $\psi:W\times\R\to W$ on $W$ generated by
the vector field $\mu\dfrac{\partial}{\partial z}$,
where $z$ is the coordinate of $\R$.
Take a smooth function $\lambda:M\to \R$ such that
$\lambda(p)>a_p$ for $p\in P$.
Then, the map
$h:W\to W-\bigcup_p\ell_p$
defined by
$h(m,t)=\psi((m,\lambda(m)), t-\lambda(m))$
is a desired diffeomorphism.
\end{proof}
Note that in this proposition $(i)$ the local finiteness
need not be assumed
and $(ii)$ the pushing-to-infinity operation can be carried out
for an arbitrary small neighborhood $N$ of $P$.
\vspace{5mm}

\noindent
{\it Proof of Theorem {\rm \ref{main}}.} 
Any connected open orientable surface $\Sigma$ can be 
constructed as follows:
First, remove from $\R^2$ a closed totally disconnected set $X$. 
($X\cup\{\infty\}$ will be the endset $\E$ of $\Sigma$,
where $\infty$ is the point of infinity of the one-point compactification of $\R^2$.)
Next, take 
an at most countable
set $Z$ in $\R^2-X$ in such a way that
for any compact set $K$ in $\R^2-X$ the intersection $Z\cap K$ is finite.
(The set of accumulation points of $Z$ in $\R^2\cup\{\infty\}$ will be $\E^*$.)
Then, for each point $q$ of $Z$, choose a small compact neighborhood
$U_q$ of $q$ in $\R^2-X$ so that they are pairwise disjoint, 
and in each $U_q$ perform a surgery to make a genus.
The resulting surface is $\Sigma$.
Observe that the whole of the above construction can be carried out
in $(\R^2-X)\times\R$.
To do so, for each $q\in Z$ 
choose a small compact neighborhood $V_q$ of $(q,0)$
in $(\R^2-X)\times\R$,
and perform ambient surgeries on $(\R^2-X)\times\{0\}$ 
inside each $V_q$. 
Thus, we obtain $\Sigma$ 
as a properly embedded submanifold of $(\R^2-X)\times\R$.
Note that $\Sigma$ separates $(\R^2-X)\times\R$ into
two connected components.

We then take a Morse function $f:(\R^2-X)\times\R\to\R$ 
so that 
\begin{itemize}
\item[(1)] $f(x,y,z)=z$ for 
$(x,y,z)\in(\R^2-X)\times[(-\infty,-1]\cup[1,\infty)]$, and that
\item[(2)] $0$ is a regular value of $f$ with $f^{-1}(0)=\Sigma$.
\end{itemize}
The existence of such $f$ follows from the above construction of $\Sigma$.
We let $\operatorname{Crit} (f)$ denote the set of critical points of $f$,
(which is a countably infinite set if $\Sigma$ has nonplanar ends).
Now, we take an increasing sequence
$\emptyset=K_0\subset K_1\subset K_2\subset\cdots$
 of codimension $0$ compact submanifolds in $\R^2-X$ 
such that $\bigcup_{i=1}^\infty K_i=\R^2-X$.
For each $p\in\operatorname{Crit} (f)$, we can
construct an injective smooth path
$c_p:[0,\infty)\to(\R^2-X)\times\R$ as follows.
Suppose $p\in (K_i-K_{i-1})\times\R$. Then,
\begin{itemize}
\item[(1)] $c_p(0)=p$,
\item[(2)] $c_p$ intersects neither $\Sigma$ nor $(\R^2-X)\times\{\pm1\}$,
\item[(3)] $c_p$ does not intersect $K_{i-1}\times\R$,
\item[(4)] for each $j\geqq i$, $c_p$ intersects
$\partial K_j\times\R$ transversely at exactly one point,
\item[(5)] $c_p(t)$ converges to a point in 
$(X\cup\{\infty\})\times\{\pm1/2\}$ as $t\to\infty$, and
\item[(6)] if $p\ne q$, then
$c_p([0,\infty))$ and $c_q([0,\infty))$ are disjoint.
\end{itemize}
We denote by $M$ the space obtained from $(\R^2-X)\times\R$
by removing $\bigcup_{p\in\operatorname{Crit}(f)}c_p([0,\infty))$.
Then, by applying Proposition \ref{push} for $W=(\R^2-X)\times\R$ and $\C=\{c_p\}$
we see that $M$ is diffeomorphic to $(\R^2-X)\times\R$.

Next, define an open subset $\Omega$ of $\R^3$
to be the union of $M$ and $\R^2\times(2,\infty)$.
Note that, by the above argument, 
$\Omega$ is diffeomorphic to $[(\R^2-X)\times\R]\cup[\R^2\times(2,\infty)]$.
Thus, by Proposition 3.2 
we can push $X\times(-\infty,2]$ to $-\infty$
with respect to the second coordinate
and obtain that $\Omega$ is diffeomorphic to $\B^3$.

Next, we extend the domain of our Morse function $f$ to $\Omega$
by
defining $f$ to be the projection to the second factor on $\R^2\times(2,\infty)$.
We let $\G$ denote the foliation on $\Omega$ 
whose leaves are connected components of the level sets of $f$.
Then, $\G$ has no singularities because all the critical points of $f$ are removed
from $\Omega$.
It is also obvious that all leaves of $\G$ are closed in $\Omega$. 
By the construction, we see that $\G$ satisfies 
the conditions (1) and (2) of Proposition 2.5.
Therefore,  
we conclude that $\G$ is diffeomorphic to a complete closed foliation 
on $\B^3$ containing $\Sigma$ as a leaf.
Theorem \ref{main} is proved.
\hfill$\square$

\section{uni-leaf foliations}
In this section
we consider the question: 
which manifold is a leaf of a complete closed uni-leaf foliation 
on the open unit ball?
This question has first been asked by Alarc\'{o}n and Forstneri\^{c} \cite{AF} in the holomorphic category.
They have shown that for any integer $n>1$,
there exists a complete closed holomorphic uni-leaf
foliation of the open unit ball in $\CC^n$
with disks as leaves.
We work in the real smooth category and prove Theorem \ref{uni}.

In order to clarify the flow of the proof of our theorem,
we first treat two simple cases:
$\Sigma_C$ and $\Sigma_J$ given in \S 1.
We will realize each of them as a leaf of a complete closed smooth uni-leaf
foliation of $\B^3$.
(Then, the full proof of Theorem \ref{uni}
will be understood as an elaboration of these cases.)

\begin{xexa}
{\rm
Let $C$, $\Sigma_C$, $\E$, $\E^\pm$ and $e$ be as in Example $(1)$ 
in \S {\rm 1}.
Through the identification of $S^2-\{e\}$ with $\R^2$,
we regard $\Sigma_C$, $\E-\{e\}$ and 
$\E^\pm-\{e\}$ as subsets of $\R^2$.
Now, put
$$\Omega=\R^3
-(\E^+-\{e\})\times[-1,\infty)
-(\E^--\{e\})\times(-\infty,1].
$$
Then, by Proposition {\rm 3.2},
$\Omega$ is diffeomorphic to $\R^3$.
We denote by $\G$ the foliation on $\Omega$ obtained by
restricting the foliation $\{\operatorname{pr}_3^{-1}(z)\}_{z\in\R}$ on $\R^3$.
Then, all leaves of $\G$ are diffeomorphic to 
$\Sigma_C$.
In fact, it is obvious when $|z|\leqq1$.
In the case when $z>1$ $($resp. $z<-1$$)$, we have $\operatorname{pr}_3^{-1}(z)\cap\Omega$
is diffeomorphic to $\R^2-(\E^+-\{e\})$
$($resp. $\R^2-(\E^--\{e\})$$)$. 
But, since we are assuming that $\E^\pm$ are homeomorphic to $\E$, 
we have that $\operatorname{pr}_3^{-1}(z)\cap\Omega$ is diffeomorphic to $\Sigma_C$ also in this case.
Finally, since $\G$ satisfies 
the conditions $(1)$ and $(2)$ of Proposition {\rm 2.5},
$\G$ is diffeomorphic to a complete closed uni-leaf 
foliation on $\B^3$, as desired.
}
\end{xexa}
 
\begin{xexa}
{\rm
We first embed Jacob's ladder $J$ in $\R^3$ as follows.
Let $H=\R^2\times\{0\}$
and $0<\varepsilon\ll1$. 
For each $n\in\Z-\{0,\pm1\}$, choose a small neighborhood $U_n$
of $(n,0)$ in $\R^2$.
We put
$W_n^+=U_n\times(-1-\varepsilon,n+\varepsilon)$ $(n\ge2)$ 
and
$W_n^-=U_{n}\times(n-\varepsilon,1+\varepsilon)$ 
$(n\le-2)$.
Inside each $W_n^\pm$ we perform an ambient surgery 
on $H$ to make a genus. 
Thus, we obtain a new surface embedded in $\R^3$
and diffeomorphic to $J$.
Hereafter, we identify this surface with $J$.
Next, we put
$\ell^+_n=\{(n,0)\}\times[-1,\infty)$ $(n\ge2)$ and
$\ell^-_n=\{(n,0)\}\times(-\infty,1]$ $(n\le-2)$.
We can take a Morse function  
$f:\R^3\to \R$ satisfying the following conditions
$($by isotoping $J$ suitably in $\bigcup W_n^+\cup\bigcup W_n^-$ if necessary$)$:
\begin{itemize}
\item[(1)] 
$0$ is a regular value of $f$ and $f^{-1}(0)=J$, 
\item[(2)] $f(x,y,z)=z$
outside the union of $W_n^+$'s
and $W_n^-$'s.
\item[(3)] 
The critical points of $f$ are $A_n^+=(n,0,-1)$, $B_n^+=(n,0,n)$
$(n\ge2)$ and
$A_n^-=(n,0,n)$, $B_n^-=(n,0,1)$ $(n\le-2)$,
whose critical values are their $z$-coordinates.
If we pass through $A_n^+$ or $A_n^-$
$($resp. $B_n^+$ or $B_n^-)$,
in the direction of increasing values of $f$,
then the level set of $f$ is modified so that a genus is created
$($resp. erased$)$.
\item[(4)] Inside each $W_n^\pm$, 
$f$ is a standard Morse function admitting a cancelling
pair of critical points.
\item[(5)] on each $\ell_n^\pm$, $f$ is strictly increasing 
with respect to $z$.
\end{itemize}
Let $\HH$ denote the foliation $($with singularity$)$ on $\R^3$
with the level sets of $f$ as leaves.
Then, every regular leaf $f^{-1}(z)$ of $\HH$
is diffeomorphic to $J$.
In fact, if $|z|<1$, it is obvious. 
If $|z|>1$, $f^{-1}(z)$ loses \lq\lq half" the number of infinite genus
in comparison with $f^{-1}(0)$.
But it still possesses infinite genus, hence, is diffeomorphic to $J$.
We can also observe that any singular leaf of $\HH$ has infinite genus.
Now, denote by $\Omega$ the set obtained from $\R^3$ by removing
the infinite family of half lines $\ell^\pm_n$.
Then, it follows from 
Proposition {\rm \ref{push}}
that
$\Omega$ is diffeomorphic to $\R^3$.
Since all the critical points are removed,
$\HH$ restricted to $\Omega$ becomes a nonsingular foliation,
say $\G$.
We will check the topology of leaves of $\G$.
First, we recall that the diffeomorphism type of $\Sigma_J$
is characterized by  being orientable and having infinite genus, one non-planar end $($say $e$$)$ and countably infinite isolated planar ends 
converging to $e$. 
Note that 
each leaf $L=f^{-1}(z)\cap\Omega$ of $\G$ is obtained 
from the leaf $H=f^{-1}(z)$ of $\HH$ by removing 
countably infinite discrete points $H\cap\ell^\pm_n$.
Therefore, if $H$ is a regular leaf $($hence diffeomorphic to $J$$)$, then 
$L$ is diffeomorphic to $\Sigma_J$.
In the case when $H$ is a singular leaf, we have to notice that
by the removal of one singular point from $H$, two punctures $($i.e. planar ends$)$ are produced on $L$.
But, anyway, planar ends of $L$ are countably infinite and anyone of them 
is isolated. Therefore, $L$ is  diffeomorphic to $\Sigma_J$ also in this case.
Consequently, all the leaves of $\G$ are diffeomorphic to $\Sigma_J$.
As a final step, we take a diffeomorphism $\Psi:\R^3\to V$,
where $V=\{(x,y,z)\in \R^3\mid (x^2+y^2)|z|<1\}$,
such that $\Psi$ leafwise preserves the one-dimensional foliation
$dx=dy=0$.
If we set $\widehat\Omega=\Psi(\Omega)$,
we can see that
the foliation $\Psi(\G)$ on $\widehat\Omega$
satisfies
the property $({\rm P})$ in \S {\rm 2},
while $\widehat\Omega$ is unbounded in the direction of $z$.
Consequently,
by Proposition {\rm 2.5}, 
we obtain a complete closed uni-leaf
foliation on $\B^3$ 
with all leaves diffeomorphic to $\Sigma_J$.
}
\end{xexa}

Now, we will proceed to 

\noindent
{\it Proof of Theorem {\rm \ref{uni}}.} 
Let $\Sigma$ be a connected open orientable smooth surface 
and $(\E,\E^*)$ the endset pair of $\Sigma$.
We use the notation in \S 1.
We assume that there exist $e$ and $Z$ as in Theorem \ref{uni}
such that
$(\E,\E^*, Z,e)$ satisfies the self-similarity property.
(Here, we should recall that if $\Sigma$ has no genus, 
then $\E^*$ and $Z$ are empty.)
Put $X=\E-\{e\}$ and $X^\pm=\E^\pm-\{e^\pm\}$.
Via $h$, we regard $\E^\pm$, $X^\pm$ and $Z^\pm$
as subsets of $\E$, $X$ and $Z$ respectively.

Now, we will start the construction of the uni-leaf foliation.
We embed $\E$ into the one-point compactification
$\R^2\cup\{\infty\}$ of $\R^2$
in such a way that $e$ is mapped to $\infty$.
From now on, we identify $e$ with $\infty$
and regard $X$ and $Z$ as subsets of $\R^2$.
We consider two cases separately.

\noindent
{\bf The case where $Z$ is empty.} 
In this case, the construction is a verbatim repetition of the one
in the case of $\Sigma_C$ ($S^2$ minus a Cantor set):
Namely, put
$$\Omega=\R^3-X^+\times[-1,\infty)-X^-\times(-\infty,1]$$
and let $\G$ denote the foliation on $\Omega$ obtained by
restricting the foliation $\{\operatorname{pr}_3^{-1}(z)\}_{z\in\R}$ on $\R^3$.
Then, by Proposition 3.2, $\Omega$ is diffeomorphic to $\R^3$ and,
by the self-similarity condition, 
all leaves of $\G$ are diffeomorphic to 
$\Sigma$.
Finally, 
by Proposition 2.5, we can conclude that $\G$ is diffeomorphic to a complete closed uni-leaf foliation on $\B^3$.

\noindent
{\bf The case where $Z$ is countably infinite.} 
Since $X-Z$ is closed in $\R^2$,
similarly as in \S 3
for each point $q$ of $Z$, we can choose a small compact 
neighborhood
$V_q$ of $(q,0)$ in the open $3$-manifold $(\R^2-(X-Z))\times\R$ 
so that they are pairwise disjoint. 
Then, perform an ambient surgery on $(\R^2-(X-Z))\times\{0\}$ 
to make a genus inside each $V_q$. 
Thus, we obtain a new surface 
as a properly embedded submanifold of $(\R^2-(X-Z))\times\R$.
Let $\widehat\Sigma$ denote this surface.
We see that the endset pair of $\widehat\Sigma$ is $(\E-Z, \E^*)$.
We may assume that 
 for each $q\in Z$ the intersection of $\widehat\Sigma$ and $\{q\}\times\R$
is a single point.

By the self-similarity property,
$X$ and $Z$ are respectively expressed as the disjoint union
of two subsets as follows:
$X=X^+\sqcup X^-$ and
$Z=Z^+\sqcup Z^-$,
with the property that
$X^\pm$ and $Z^\pm$ are respectively homeomorphic 
to $X$ and $Z$.
We put 
$A^+=(X^+-Z^+)\times[-1,\infty)$,
$A^-=(X^--Z^-)\times(-\infty,1]$ and 
$$
O=\R^3-A^+-A^-.
$$
Note that $X^+-Z^+$ and $X^--Z^-$ are closed in $\R^2$,
and hence $A^+$ and $A^-$ are closed in $\R^3$.
We number the elements of $Z$ arbitrarily:
$Z^+=\{q_n\mid n=2,3,4,\cdots\}$ and 
$Z^-=\{q_n\mid n=-2,-3,-4,\cdots\}$.
We then take a Morse function $f:O\to\R$ 
so that 
\begin{itemize}
\item[(1)]  $f^{-1}(0)=\widehat\Sigma$,
\item[(2)]  $\operatorname{Crit} (f)$ consists of the following points:\\
$(q_n,-1-\frac1n)$ and $(q_n,n)$ for each $n=2,3,4\cdots$,\\
$(q_n,n)$ and $(q_n,1-\frac1n)$ for each $n=-2,-3,-4\cdots$.
\item[(3)]  For each $p\in\operatorname{Crit} (f)$, the value $f(p)$ is
 the $z$-coordinate of $p$,
\end{itemize}
 Let $W^+_q=D^+_q\times I^+_q$ ($q\in Z^+$) be 
a compact product neighborhood of the segment $\{q\}\times [-1-\frac1n,n]$ 
in $O$,
where $D^+_q$ is a closed disk in $\R^2$ 
centered at $q$ 
such that 
$D^+_q\cap X=\{q\}$ 
and $[-1-\frac1n,n]\subset I^+_q\subset\R$.
Similarly,
let $W^-_q=D^-_q\times I^-_q$ ($q\in Z^-$) be a compact product neighborhood of
$\{q\}\times [n,1-\frac1n]$ in $O$,  
where $D^-_q$ is a closed disk in $\R^2$ 
centered at $q$ 
such that 
$D^-_q\cap X=\{q\}$ 
and $[n,1-\frac1n]\subset I^-_q\subset\R$.
We choose the sets $D^+_q$ ($q\in Z^+$) and $D^-_q$ ($q\in Z^-$) 
so as to be pairwise disjoint.
\begin{itemize}
\item[(4)] Inside each $W^+_q$ ($q\in Z^+$) or $W^-_q$ ($q\in Z^-$),
$f$ is conjugate to the standard Morse function which admits a standard 
canceling pair of critical points, the one which has a smaller 
$z$-coordinate is of index $1$ and the other is of index $2$.
\item[(5)] The lines $\{q\}\times[-1-\frac1n,\infty)$ ($q\in Z^+$) and 
$\{q\}\times(-\infty,1-\frac1n]$ ($q\in Z^-$) are transverse to the level sets
of $f$ everywhere except at critical points,
\item[(6)] $f(x,y,z)=z$ outside the union of $W^+_q$'s ($q\in Z^+$) and 
$W^-_q$'s ($q\in Z^-$).
\end{itemize}
Then, the family of the level sets of $f$ defines a singular foliation
on $O$.
The singularities are the critical points of $f$.
We see that
each level set contains at most one critical point.
We can also observe that for each $z\in\R$ 
the endset pair $(\E_z, \E_z^*)$ of the level set $f^{-1}(z)$ 
is identified with:
$(\E-Z, \E^*)\times\{z\}$ 
if $|z|\leqq 1$,
$(\E^+-Z^+, \E^{+*})\times\{z\}$ if $z>1$, and
$(\E^--Z^-, \E^{-*})\times\{z\}$ if $z<-1$.
By the self-similarity property, all of these are homeomorphic to
$(\E-Z,\E^*)$.

As the next step, we define
$$C=\bigcup_{q_n\in Z^+}\{q_n\}\times\left[-1-\frac1n,\infty\right) 
\cup\bigcup_{q_n\in Z^-}\{q_n\}\times\left(-\infty,1-\frac1n\right]
$$
and
$$
\Omega=O-C.
$$
Then, by Propositions 3.1 and 3.2
we see that $\Omega$ is diffeomorphic to $\R^3$.
Each level set $L_z=f^{-1}(z)\cap\Omega$ of $f|_\Omega$
is obtained from $f^{-1}(z)$ by deleting the points of intersection 
with $C$.
Since all the critical points of $f$ are removed by this deletion,
every $L_z$
is now a non-singular smooth surface.
Let $\G$ be the foliation on $\Omega$ thus obtained.
Here, observe that if the point of intersection of $f^{-1}(z)$ and
 $\{q\}\times I$ ($q\in Z$, and $I$ is either $[-1-\frac1n,\infty)$
or $(-\infty,1-\frac1n]$) is not a critical point, then the deletion
 yields one puncture (or, one planar end) on $f^{-1}(z)$,
while if the point of intersection is a critical point,
then the deletion yields two punctures (or, two planar ends).
Now, let $Z_z$ be the set of all ends 
of $L_z$ newly produced by these deletions.
Then, the endset pair of $L_z$ is expressed as 
$(\E_z\cup Z_z, \E_z^*)$,
where $(\E_z, \E_z^*)$ is the endset pair of $f^{-1}(z)$. 
Since, as remarked above, each $f^{-1}(z)$ contains 
at most one critical point,
it follows from the property (2) of $f$
that $Z_z$ is identified with:
$Z$ if $|z|\leqq 1$,
the union of $Z^+$ and $F_z$ if $z>1$,
the union of $Z^-$ and $F_z$ if $z<-1$,
where $F_z$ is a (possibly empty) finite subset of $\R^2-X$.
(Supplementary explanation:
 If $z\geqq2$, $f^{-1}(z)$ does not intersect 
$\{q\}\times(-\infty,1-\frac1n]$ for any $q\in Z^-$. 
So, in this case, $F_z$ is either a singleton or empty
depending on whether there exists a critical point on
$f^{-1}(z)\cap\{q\}\times[-1-\frac1n,\infty)$ for some $q\in Z^+$.
If $1<z<2$, we see that $f^{-1}(z)$ intersects 
$\{q\}\times(-\infty,1-\frac1n]$ for at most finitely many $q\in Z^-$.) 
Therefore,
$(\E_z\cup Z_z, \E_z^*, Z_z, \infty)$,
is identified with:
$(\E, \E^*, Z, \infty)\times\{z\}$ if
$|z|\leqq 1$,
$(\E^+\cup F_z, \E^{+*}, Z^+\cup F_z, \infty)\times\{z\}$ if $z>1$, and
$(\E^-\cup F_z, \E^{-*}, Z^-\cup F_z, \infty)\times\{z\}$ if $z<-1$.

\begin{Lem}\label{EZ}
If $F$ is a finite subset of $\R^2-X$,
then there is a homeomorphism $h:\E\cup F\to\E$
such that $h$ is the identity on $\E^*$ and that $h(Z\cup F)=Z$.
\end{Lem}

\begin{proof}
Let $F$ be $\{x_1,\cdots,x_r\}$.
Take a point $p$ in $\E^*$
and any sequence $\{p_i\}_{i=1}^\infty$ in $Z$
converging to $p$.
We define a bijection $h:\E\cup F\to\E$
by:
$h(x_k)=p_k$ for $k=1,\cdots,r$,
$h(p_i)=p_{r+i}$ for $i\geqq1$, and
$h$ is the identity otherwise.
Then, the continuity of $h$ easily follows.
\end{proof}
\vspace{2mm}

By this lemma and the self-similarity property, 
the $4$-tuple $(\E_z\cup Z_z, \E_z^*,
\break
Z_z, \infty)$ 
for the leaf $L_z$ is homeomorphic to
$(\E, \E^*, Z, \infty)$ for every $z\in\R$.
Hence, we can conclude that all the leaves $L_z$ of $\Omega$
is diffeomorphic to $\Sigma$. 

As the final step, we will transform the foliation $(\Omega, \G)$  
so that the resulting foliation satisfies the property (P).
To do so, take an arbitrary point $(x_0,y_0)$ from $\R^2-X$ 
and put $V=\{(x,y,z)\in \R^3\mid ((x-x_0)^2+(y-y_0)^2)|z|<1\}$.
Next, choose any diffeomorphism $\Psi:\R^3\to V$
which preserves leafwise the vertical foliation $dx=dy=0$.
Then, we can see that
the foliation $\Psi(\G)$ on $\Psi(\Omega)$
satisfies
the property $({\rm P})$ in \S {\rm 2},
while $\Psi(\Omega)$ is unbounded in the direction of $z$.
Therefore,
by Proposition {\rm 2.5}, 
we obtain a complete closed uni-leaf
foliation on $\B^3$ 
with leaves diffeomorphic to $\Sigma$.
This completes the proof of Theorem \ref{uni}.
\hfill$\square$
\vspace{3mm}

\noindent
{\bf Examples of surfaces with the self-similarity property.}
In the case of planar surfaces, $\E^*$ and $Z$ are empty,
hence, to check the self-similarity, we have only to show
that $\E_1\vee_{e_1=e_2}\E_2$ is homeomorphic to $\E$ for some $e\in\E$,
where $(\E_i,e_i)$, $i=1,2$, are copies of $(\E,e)$.
The following surfaces satisfy such a property:
$\R^2$, $\R^2$ minus a discrete closed infinite set, 
$\R^2$ minus a Cantor set, and $S^2$ minus a Cantor set.

In the case of nonplanar surfaces,
there are also many examples.
Here, we give one family of surfaces $\Sigma(r)$, $r\in\N$
$($Example $(2)$ in \S $1$ is $\Sigma(1))$.

\noindent
The endset pair $(\E, \E^*)$ of $\Sigma(r)$ is described as follows:

\begin{align*}
\E&=\{e, e_{i_1},e_{i_1i_2},\cdots,e_{i_1i_2\cdots i_r}\mid 
i_k\in\N, 1\leqq k\leqq r\},\\
\E^*&=\{e, e_{i_1},e_{i_1i_2},\cdots,e_{i_1i_2\cdots i_{r-1}}\mid 
i_k\in\N, 1\leqq k\leqq r-1\},\\
Z&=\{e_{i_1i_2\cdots i_r}\mid 
i_k\in\N, 1\leqq k\leqq r\}.
\end{align*}

Let $\E^{(\ell)}$ denote the $\ell$-th derived set of $\E$.
Then, for $1\leqq\ell\leqq r-1$, 
\begin{align*}
\E^{(\ell)}&=\{e, e_{i_1},e_{i_1i_2},\cdots,e_{i_1i_2\cdots i_{r-\ell}}\mid 
i_k\in\N, 1\leqq k\leqq r-\ell\},
\end{align*}
and $\E^{(r)}=\{e\}$.
For each $1\leqq k\leqq r$,
$e_{i_1i_2\cdots i_k}$ converges to $e_{i_1i_2\cdots i_{k-1}}$ as $i_k\to\infty$ while $i_1,i_2,\cdots,i_{k-1}$ being fixed,
and $e_{i_1}$ converges to $e$ as $i_1\to\infty$.

Now, put
\begin{align*} 
\E^+&=\{e, e_{i_1},e_{i_1i_2},\cdots,e_{i_1i_2\cdots i_r} \mid 
\textrm{$i_1$ is even and $i_2,\cdots,i_r$ are arbitrary}\}, \\
\E^-&=\{e, e_{i_1},e_{i_1i_2},\cdots,e_{i_1i_2\cdots i_r}\mid 
\textrm{$i_1$ is odd and $i_2,\cdots,i_r$ are arbitrary}\}, \\
\E^{+*}&=\E^+\cap\E^*,\ 
\E^{-*}=\E^-\cap\E^*,\ 
Z^+=\E^+\cap Z,\ 
Z^-=\E^-\cap Z,\ 
e^+=e^-=e, 
\end{align*}
and $h=id: \E^+\vee_{e^+=e^-} \E^-\to \E$.
Then, the $4$-tuple $(\E,\E^*,Z,e)$ of the surface 
$\Sigma(r)$ satisfies the self-similarity.
\vspace{4cm}

\begin{figure}[htbp]
\centering
\includegraphics[height=10cm]{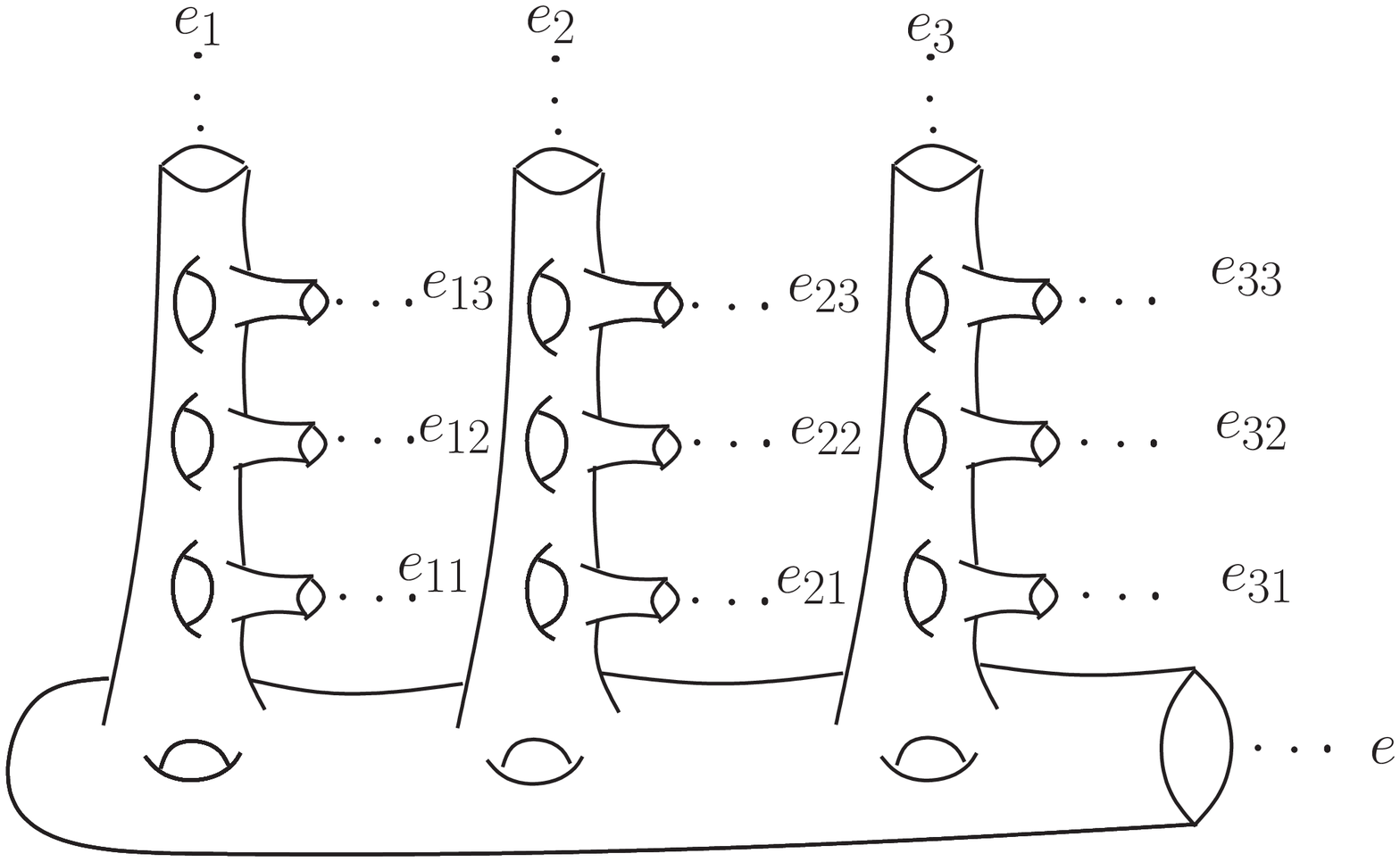}
      \caption{$\Sigma(2)$}
\end{figure}
\vspace{2mm}

\begin{xque}
{\rm
List up all the open orientable surfaces whose endsets satisfy 
the self-similarity property.
}
\end{xque} 

\begin{xque}
{\rm
Can a surface which does not satisfy the self-similarity property
be realized as a leaf of a uni-leaf foliation on $\B^3$?
}
\end{xque}

\section{Higher dimensional leaves}
In this section we consider the case of higher dimensional leaves.
We give two results.
The first one is the following

\begin{Thm}
Let $M$ be a simply connected open $n$-manifold, $n\geq3$,
with a smooth foliation $\F$ by leaves diffeomorphic to $\R^{n-1}$.
Then, $\F$ is smoothly conjugate to a complete closed 
$($and necessarily, uni-leaf$)$ foliation on $\B^n$.
\end{Thm}
This result is kindly taught by the referee to the authors.

\begin{proof}
Let $M$ and $\F$ be as in the hypothesis of the theorem.
Then, 
it follows from the deep result of 
Palmeira \cite{P}
that there are an open set $\Omega$ of $\R^n$
diffeomorphic to $\R^n$
and a diffeomorphism $h:M\to\Omega$ such that $h(\F)$ coincides
with the restriction to $\Omega$ of the foliation of $\R^n$
by horizontal hyperplanes.
Here, we may assume that $\Omega$ satisfies the property (1) of 
Proposition 2.5.
In fact, if $\operatorname{pr}_n(\Omega)$ is bounded,
it suffices to replace $\Omega$ with 
$(id_{\R^{n-1}}\times\varphi)(\Omega)$,
where $\varphi$ is an arbitrary diffeomorphism
from the open interval $(\inf\operatorname{pr}_n(\Omega),\sup\operatorname{pr}_n(\Omega))$ to $\R$.
Clearly, $h(\F)$ satisfies the property (2) of 
Proposition 2.5.
Thus, by Proposition 2.5, 
there exists a diffeomorphism $\Phi$ from $\Omega$ to $\B^n$
such that $\Phi(h(\F))$ is a complete closed foliation on $\B^n$.
\end{proof}

Let $\overline{\B}^n$ denote the closed unit $n$-ball, and
$\operatorname{pr}_i$ the projection from a product space
to its $i$-th factor. The next result in this section is the following

\begin{Thm}\label{higher}
Let $n\geqq 3$. Suppose that $F$ is a connected compact $(n-1)$-dimensional smooth 
submanifold of $\overline{\B}^{n-1}\times\R$
such that 
$F\cap(\partial\overline{\B}^{n-1}\times\R)
=\partial\overline{\B}^{n-1}\times\{0\}=\partial F$
and that $F$ is transverse to 
$\partial\overline{\B}^{n-1}\times\R$ at $\partial F$.
Let $E$ be a closed subset of $F$ satisfying that
\begin{itemize}
\item[(1)] $F-E$ is connected,
\item[(2)] $E$ contains $\partial F$, and that
\item[(3)]there exists a neighborhood $U$ of $E$ in $F$ such that
$\operatorname{pr}_1:\overline{\B}^{n-1}\times\R\to\overline{\B}^{n-1}$
 maps $U$ diffeomorphically to $\operatorname{pr}_1(U)$
and that $\operatorname{pr}_1^{-1}\operatorname{pr}_1(U)\cap F=U$.
\end{itemize}
Then, there is a codimension $1$ complete closed smooth 
foliation of $\B^n$ with a leaf diffeomorphic to $F-E$.
\end{Thm}

\begin{proof}
The proof is essentially the same as the one in the surface case.
Let $n$, $F$ and $E$ be as above.
We take a Morse function $f:(\overline{\B}^{n-1}-\operatorname{pr}_1(E))\times\R\to\R$ 
so that 
\begin{itemize}
\item[(1)] $f=\operatorname{pr}_2$ on $(\overline\B^{n-1}-\operatorname{pr}_1(E))\times[(-\infty,-1]\cup[1,\infty)]$, and that
\item[(2)] $0$ is a regular value of $f$ with $f^{-1}(0)=F-E$.
\end{itemize}
Next, we take an exhausting sequence $\{K_i\}$
 of codimension $0$ compact submanifolds in 
$\overline{\B}^{n-1}-\operatorname{pr}_1(E)$,
and a family of injective smooth paths
$c_p:[0,\infty)\to(\overline\B^{n-1}-\operatorname{pr}_1(E))\times\R$, 
$p\in\operatorname{Crit} (f)$,
satisfying the same six conditions with the ones in \S 3.
Then, 
$M=(\overline\B^{n-1}-\operatorname{pr}_1(E))\times\R-\bigcup_{p\in\operatorname{Crit}(f)}c_p([0,\infty))$ is diffeomorphic to $(\overline\B^{n-1}-\operatorname{pr}_1(E))\times\R$,
and $\Omega=M\cup(\B^{n-1}\times(2,\infty))$ is diffeomorphic to $\B^n$.
Finally, by exactly the same argument given in \S 3
we complete the proof of Theorem \ref{higher}.
\end{proof}


\begin{xrem}
{\rm
For example, we may take as $E-\partial F$ 
the Whitehead continuum, the Menger sponge, and so on.
}
\end{xrem} 

\section{Higher codimensions}

\begin{Prop} \label{codim}
Let $q$ and $q'$ be positive integers such that $1\leqq q<q'$.
Given a connected $p$-dimensional manifold $L$,
if there is a codimension $q$ complete closed smooth 
foliation on $\B^{p+q}$ with a leaf diffeomorphic to $L$,
then,
there is a codimension $q'$ complete closed smooth 
foliation on $\B^{p+q'}$ with a leaf diffeomorphic to $L$.
\end{Prop}

\begin{proof}
Suppose $\F$ is a codimension $q$ complete closed smooth 
foliation on $\B^{p+q}$ with a leaf diffeomorphic to $L$.
Then the foliation on $\B^{p+q}\times \B^{q'-q}$ defined by
$F\times\{z\}$ ($F\in\F$, $z\in \B^{q'-q}$) as leaves
is a codimension $q'$ complete closed smooth foliation
and has a leaf diffeomorphic to $L$.
Since $\B^{p+q}\times \B^{q'-q}$ is diffeomorphic to
$\B^{p+q'}$ by a quasi-isometric diffeomorphism, 
the conclusion follows.
\end{proof}

Combining Proposition \ref{codim} with Theorem \ref{main} and
Theorem \ref{higher}, we obtain

\begin{Thm}
Let $L$ be $\Sigma$ in Theorem {\rm \ref{main}} or
$F-E$ in Theorem {\rm \ref{higher}}, and let $p=\dim L$.
Then, for any positive integer $q$,
there is a codimension $q$ complete closed smooth 
foliations on the open unit ball $\B^{p+q}$
having $L$ as a leaf.
\end{Thm}

Similarly, by Proposition \ref{codim} and Theorem \ref{uni}
we have

\begin{Thm}
Let $L$ be $\Sigma$ in Theorem {\rm \ref{uni}}.
Then, for any positive integer $q$,
there is a codimension $q$ complete closed smooth 
uni-leaf foliation on the open unit ball $\B^{3+q}$
having $L$ as a leaf.
\end{Thm}

\subsection*{Acknowledgements}
The authors would like to thank the referee for many invaluable
suggestions, which have made the manuscript much more readable.
They also thank Ryoji Kasagawa and Atsushi Sato 
for their interest in our work and many helpful comments.

\end{document}